\documentclass[12pt,reqno]{amsart}
\usepackage{graphicx}
\usepackage{amsfonts}
\usepackage{amssymb}
\usepackage{amsthm}
\usepackage[centertags]{amsmath}
\usepackage{newlfont}

\usepackage{color}
\usepackage{graphicx,color} 
\usepackage{amsmath, amssymb, graphics}

\def\r{\mathbb R}
\def\l{\mathbf L}
\def\c{\mathbb C}
\def\h{\mathbb H}
\def\t{\mathbf t}
\def\n{\mathbf n}
\def\b{\mathbf b}
\def\e{\mathbf E}

\newtheorem{theorem}{Theorem}[section]
\newtheorem{corollary}[theorem]{Corollary}

\newtheorem{proposition}[theorem]{Proposition}
\theoremstyle{definition}
   \newtheorem{definition}[theorem]{Definition}
 \newtheorem{remark}[theorem]{Remark}
\begin{document}

%%%%%%%%%%%%%%%%%%%%%%%%%%%%%%%%%%%%%%%%%%%%%%%%%%%%%%%%%%
\title[On the duality between  minimal surfaces and maximal surfaces]{On the duality between rotational minimal surfaces and maximal surfaces}
\author{Rafael L\'opez}
 \email{ rcamino@ugr.es}
 \address{Departamento de Geometr\'{\i}a y Topolog\'{\i}a\\ Instituto de Matem\'aticas (IEMath-GR)\\
 Universidad de Granada\\
 18071 Granada, Spain}
\thanks{Rafael L\'opez is the corresponding author.}

 \author{Seher Kaya}
\email {seher.kaya@ankara.edu.tr}
   \address{Department of Mathematics, Ankara University\\ Ankara, Turkey}

\begin{abstract}
We investigate the duality between  minimal surfaces in Euclidean space and maximal surfaces in   Lorentz-Minkowski space in  the family of rotational surfaces. We study if the dual surfaces of two congruent rotational minimal (or maximal) surfaces  are congruent. We show that  in the duality process by means of   a one-parameter group of rotations, it appears the family of Bonnet minimal (maximal) surfaces  and the Goursat transformations.

\end{abstract}


\maketitle

\section{Introduction}

There is a correspondence, known as duality, between minimal surfaces in Euclidean space $\e^3$ and maximal surfaces in Lorentz-Minkowski space $\l^3$.  This correspondence assigns a maximal surface to each minimal surface   and {\it vice-versa} and it was introduced by Calabi for minimal surfaces and maximal surfaces expressed as graphs on a simply-connected domain \cite{ca}.  A similar correspondence between both families of surfaces also appeared on \cite{gu,lls} where  the duality is now  defined in terms of  the isotropic curve that determines the surface and finally, it was proved in \cite{le1} that both methods are  equivalent. We describe this correspondence in terms of complex analysis and the isotropic curve. If $X:\Omega \rightarrow\e^3$ is a conformal minimal surface defined on a simply-connected domain $\Omega$ of the complex plane $\c$  and $z$ is the conformal parameter, then the complex curve $\phi:\Omega\rightarrow\c^3$ defined by  $\phi(z)=2X_z=(\phi_1,\phi_2,\phi_3)$ is holomorphic and satisfies the isotropy relation $\langle \phi,\phi \rangle =\phi_1^2+\phi_2^2+\phi_3^2=0$.  If we now define $\psi:\Omega\rightarrow\c^3$ by   $(\psi_1,\psi_2,\psi_3)=(-i\phi_1,-i\phi_2,\phi_3)$, then $\psi_1^2+\psi_2^2-\psi_3^2=0$ and consequently this  defines a  maximal surface $X^\flat:\Omega\rightarrow\l^3$  by setting $X^\flat(z)=\Re\int^z\psi(z) dz$. This process has its converse:  if $X:\Omega \rightarrow\l^3$ is a conformal maximal surface  and $\psi(z)=2X_z=(\psi_1,\psi_2,\psi_3)$, then the complex curve $\phi=(i\psi_1,i\psi_2,\psi_3)$  satisfies $\langle\phi,\phi\rangle=0$ and  defines a minimal surface   $M^\sharp$ in  $\e^3$ by means of $X^\sharp(z)=\Re\int^z \phi(z) dz$.  Furthermore, and up to   translations of the ambient space,  we have $M=(M^\sharp)^\flat$.  If $\mbox{Min}$ and $\mbox{Max}$ denote the family of minimal surfaces of $\e^3$ and the  maximal surfaces of $\l^3$, respectively, we have established two maps
$$\flat:\mbox{Min}\rightarrow\mbox{Max},\quad\quad\sharp:\mbox{Max}\rightarrow\mbox{Min}$$ 
with the property that $\flat\circ\sharp$ and $\sharp \circ\flat$ are the identities in $\mbox{Max}$ and $\mbox{Min}$ respectively. We say that $M^\flat$ (or $M^\sharp$) is the {\it dual surface} of $M$ (also named  in the literature as twin surface \cite{gu,le2,ma}). This process of duality  has been generalized in other  ambient spaces: see for example,  \cite{aa,le2,pa}

In this paper we are interested in a problem posed by Araujo and Leite in \cite{al} that asks  whether  the dual surfaces of two congruent minimal (or maximal) surfaces are also congruent. We precise the terminology.  We say that two surfaces $M_1$ and $M_2$  of $\e^3$ (or $\l^3$) are congruent, and we denote by $M_1\simeq M_2$, if there is an orientation-preserving isometry of the ambient space taking one of the surface onto the other. Here we also suppose that this relation $\simeq$ is up to an automorphism on  $M_1$ and $M_2$ and up to dilations of the ambient space because dilations preserve the zero mean curvature property.  In $\e^3$ the set of congruences preserving the orientation is $SO(3)$ and in $\l^3$ is $SO(2,1)$.     Then   the problem can be formulated as follows: 

{\bf Problem 1.} If $M_1\simeq M_2$ are  two minimal surfaces of $\e^3$, does   $M_1^\flat\cong M_2^\flat$ hold?

A similar question can be posed for maximal surfaces. Surprisingly the answer is `not in general' and   there are many congruent minimal surfaces whose dual surfaces are not congruent.  The process to consider is the following. Let $M\subset  \e^3$ be a minimal surface and  the congruence class of $M$, namely,  $\{T(M):T\in SO(3)\}$, then compute $\{T(M)^\flat:T\in SO(3)\}$  and in this set we consider the equivalence relation by congruences. If $M_T^\flat/\simeq$ stands for the corresponding quotient space, we want to determine this set. Similarly, we use the notation $M_T^\sharp/\simeq$. In \cite{al} the authors study this problem in the the case that $M$ is the Enneper surface, the Scherk surface and  the catenoid. 

In this paper we focus how the geometric properties of a minimal (or maximal) surface can transform  to its dual surface and we pay our attention for rotational  surfaces with the next  question:

{\bf Problem 2.} Is the dual surface rotational     of a rotational  minimal (or maximal) surface?

Let us observe  that  the maps $\flat$ and $\sharp$ do not carry any geometrical information of the surface because in the definition of a dual surface only it is involved the complex coordinates of the surface. We point out that the class of rotational surfaces of $\l^3$ is richer than in the Euclidean case because in $\l^3$ there exist three types of rotational  surfaces according the causal character of the axis of revolution. If we restrict to maximal surfaces, there are three types of non-congruent rotational maximal surfaces, named,  elliptic catenoid, hyperbolic catenoid  and parabolic catenoid when the rotational axis is timelike, spacelike or lightlike, respectively \cite{ko}.  
 It was proved in   \cite{al} that for an Euclidean catenoid $M$, the quotient space $M_T^\flat/\simeq$ has the topology of the closed interval $[0,1]$, obtaining a Bonnet maximal surface for $t<1$ and the hyperbolic catenoid for $t=1$.  In Section \ref{s-dual} we recover this result by 
  describing explicitly the dual surfaces of the rotational  maximal surfaces of $\l^3$. Then we take an Euclidean catenoid $C$ with axis $(0,0,1)$ and we consider the dual surfaces of the Euclidean catenoids obtained by rotation $C$ about a line orthogonal to the axis of $C$.  In Section \ref{s-dual2}  we consider each one the three catenoids $C$ of $\l^3$ and we deform by rotations about an axis with different causal character than the one of $C$. We will prove in Thms. \ref{t51}, \ref{t52} and \ref{t53} that the dual surfaces belong to the Bonnet family of minimal surfaces up to a Goursat transformation or it is the Enneper surface.

%%%%%%%%%%%%%%%%%%%%%%%%%%%%%%%%%%%%%%%%
\section{Duality of minimal/maximal surfaces }
%%%%%%%%%%%%%%%%%%%%%%%%%%%%%%%%%%%%%%%%%%

Let $\r^3$ be the vector space where $(x,y,z)$ stands for the canonical coordinates. We endow $\r^3$ with the  metric $ds^2=dx^2+dy^2+\epsilon dz^2$, with $\epsilon=1$ for the Euclidean metric and $\epsilon=-1$ for the Lorentzian metric, obtaining  the Euclidean space $\e^3$ ($\epsilon=1$) and the   Lorentz-Minkowski space $\l^3$ ($\epsilon=-1$).  Let $X:M\rightarrow(\r^3,ds^2)$ be a conformal immersion of a  (connected oriented) surface $M$  with  $X=X(z)$ and  $z=u+iv\in \c$ stands for a  conformal parameter. In case $\epsilon=-1$, we are assuming that the induced metric on $M$ via $X$ is Riemannian, that is, $(M,ds^2)$ is a spacelike surface in $\l^3$. Suppose that the immersion has zero mean curvature at every point and we say that $M$ is a minimal surface ($\epsilon=1)$ or  a maximal surface ($\epsilon=-1$). This is equivalent that the immersion $X$ is harmonic and this guarantees that the curve $\phi=\phi(z):M\rightarrow\c^3$ defined by 
$$\phi(z)=(\phi_1,\phi_2,\phi_3)=2\frac{dX}{dz}$$
is holomorphic. Therefore $\phi$ satisfies $\phi_1^2+\phi_2^2+\epsilon \phi_3^2=0$, which means  that $\phi$ lies on the complex null cone (resp. Lorentzian complex null cone) of $\c^3$ if $\epsilon=1$ (resp. $\epsilon=-1$). The curve $\phi$ is called the isotropic curve of the immersion $X$. The induced metric on the surface $M$ reads as $ds^2=|\phi_1|^2+|\phi_2|^2+\epsilon |\phi_3|^2\not\equiv 0$ and then the surface is  obtained by $X(z)=X(z_0)+\Re \int^z_{z_0} \phi(z) dz $ for any curve  connecting a given point $z_0\in M$ and $z$. The integral does not depend on the curve  which is equivalent to  $\Re\int_\gamma\phi(z) dz=0$ for any closed curve $\gamma$ in $M$ and we say that $\phi $ has no real periods. Letting
\begin{equation}\label{wei}
g=\frac{\phi_3}{\phi_1-i\phi_2},\quad\quad \omega=(\phi_1-i\phi_2)dz,
\end{equation}
 the pair $(g,\omega)$ is called the Weierstrass representation of $X$, where   $g$ is a meromorphic function on $M$ and  $\omega$ is a holomorphic $1$-form on $M$. The parametrization $X$ is now   
 \begin{equation}\label{weierstrass}X(z)=X(z_0)+\Re\int_{z_0}^z \left(\frac12 (1-\epsilon g^2)\omega, \frac{i}{2}(1+\epsilon g^2)\omega, g\omega\right)
\end{equation}
and the metric is $ds=|\omega|(1+|g|^2)/2|dz|$. In order to distinguish  minimal and maximal surfaces,  we will denote by $\phi$ the isotropic curve a minimal surface in $\e^3$ and by $\psi$ for a maximal surface in $\l^3$.  If the context is known, we do not explicit if $M$ denotes a minimal or a maximal surface. 

\begin{definition} \begin{enumerate}
\item Let $M$ be a minimal conformal surface in $\e^3$ and let  $\phi$ be its isotropic curve. The dual surface 
 of $M$ is the maximal surface $M^\flat$ of $\l^3$ whose isotropic curve is $\psi=(-i\phi_1,-i\phi_2,\phi_3)$. Equivalently, if the  Weierstrass representation of $M$ is $(g,\omega)$, then the one of $M^\flat$ is $(g^\flat,\omega^\flat)=(ig,-i\omega)$.
\item Let $M$ be a maximal conformal surface in $\l^3$ and let $\psi$ be its isotropic curve. The dual surface 
 of $M$ is the minimal surface $M^\sharp$ of $\e^3$ whose isotropic curve is $\phi=(i\psi_1,i\psi_2,\psi_3)$. Equivalently, if the  Weierstrass representation of $M$ is $(g,\omega)$, then the one of $M^\sharp$ is   $(g^\sharp,\omega^\sharp)=(-ig,i\omega)$. 
\end{enumerate}
\end{definition}

\begin{remark}\label{remark2} If $(g,\omega)$ is the Weierstrass representation of a minimal (or maximal) surface $M$, then $(ig,\omega/i)$ is the Weierstrass representation of a minimal (maximal) surface that is nothing $M$ rotated $\pi/2$ about the $z$-axis. As the dual surface  $M^\flat$  is $(ig,-i\omega)$, then $(g,\omega)$ is $M^\flat$ after a $\pi/2$-rotation about the $z$-axis. We conclude that up to rotations of the ambient space,   the Weierstrass representation of $M^\flat$ coincides with the one of $M$ (similarly for the dual surface of a maximal surface). As a consequence  the duality process consists simply into to consider the same Weierstrass data with different parametrizations of the surface according to (\ref{weierstrass}).
\end{remark}

We immediately find that the third component of the dual surface is preserved up to vertical translations.  We study the duality process under some transformations of the ambient space.

\begin{proposition}\label{pr1}
\begin{enumerate} 
\item If $T$ is a translation of the ambient space, then $T(M)^\flat\simeq M^\flat$ and $T(M)^\sharp\simeq M^\sharp$.
\item If $R_\theta$ is the rotation with respect to the axis $(0,0,1)$, then 
$R_\theta(M)^\flat\simeq M^\flat$ and $R_\theta(M)^\sharp\simeq M^\sharp$.
\item If $h_\lambda$ is a dilation   of ratio $\lambda>0$, then 
$h_\lambda(M^\flat)\simeq  M^\flat$ and $h_\lambda(M^\sharp)\simeq M^\sharp$.
\item If $M_\theta$ denotes the associate surface of $M$ corresponding to the parameter $\theta\in\r$, then    $(M_\theta)^\flat\simeq (M^\flat)_\theta$ and $(M_\theta)^\sharp\simeq (M^\sharp)_\theta$.
\end{enumerate}
\end{proposition} 

\begin{proof} The three first properties are immediate by observing that if  $(g,\omega)$ is the Weierstrass representation of $M$, then the one of the translation of $M$ is the same, the one of $R_\theta(M)$ is $(e^{i\theta}g,e^{-i\theta}\omega)$ and the one of $h_\lambda(M)$ is $(g,\lambda\omega)$.
For the fourth item, recall that the associate surfaces of  $M$ are the minimal (resp. maximal) surfaces $M_\theta$ whose isotropic curve is $ e^{i\theta}\phi$ (resp. $e^{i\theta}\psi$), $\theta\in\r$. The surface $M_{\pi/2}$ is called the adjoint of $M$. Then the proof of item 4 follows now  immediately.
\end{proof}

As a consequence of this result, the answer to the question posed in Problem 1 is yes when the congruence is a translation,  a rotation about the $z$-axis and also by a dilation of the ambient space.  

   \begin{remark} The duality process   can be also introduced  by multiplying the coordinates  of the isotropic curve  by the unit imaginary number $i$ and/or exchanging the order of the coordinates of the immersion. However all are  equivalent up to congruences and associate surfaces. For example, for a minimal  surface $M$,  the dual   surface  in \cite{lls} is given by  the assignment  $(\phi_1,\phi_2,-i\phi_3)$ as isotropic curve. If we calculate its adjoint surface, we multiply by $i$, obtaining $(i\phi_1,i\phi_2,\phi_3)$ which is the surface $M^\sharp$ of our definition. In \cite{al}, the isotropic curve of the dual surface of $M$ is $(-i\phi_2,i\phi_1,\phi_3)$, which is a $\pi/2$-rotation about the $z$-axis of $M^\sharp$.
     \end{remark}

%%%%%%%%%%%%%%%%%%%%%%%%%%%%%%%%%%%%%%%%%%

 \section{The Bj\"{o}rling  problem and rotational maximal surfaces} \label{appendix}
 
 In  this paper  we need to know the isotropic curves   that define a rotational  minimal or maximal surface. For this purpose we utilize the Bj\"{o}rling problem.  The  Bj\"{o}rling  problem consists into finding   a   surface with zero mean curvature  containing a given real analytic curve $\alpha$ called the core curve, and  a prescribed unit analytic normal vector field $V$ along $\alpha$. In case that the ambient space is $\l^3$,  the surface that we are looking for is spacelike and $V$ must be a timelike vector field. The surface that solves  the problem parametrizes as    
\begin{equation}\label{bj}
X(u,v)=\Re\left(\alpha(z)-i\epsilon \int\limits_{z_0}^z V(w)\times \alpha'(w)\ dw \right),
\end{equation}
where $z_0\in I$ is fixed,  $z=u+iv\in\Omega$ and $\times$ stands for the cross-product in each  space,  see \cite{acm,ni}. The solution defined  in (\ref{bj}) considers  the interval $I$ as $I\times\{0\}\subset\mathbb{C}$ and  by analyticity, the functions $\alpha$ and $V$ have holomorphic extensions $\alpha(z)$ and $V(z)$ in a simply-connected domain $\Omega\subset\c$ that contains $I\times\{0\}$.  The surface obtained, called the Bj\"{o}rling surface, is unique under the condition that  $\alpha$ is the parameter curve $v=0$ and its isotropic curve is $X_z=\alpha'(z)-i\epsilon V(z)\times\alpha'(z)$.  

It is usual in the literature to consider the rotational maximal surfaces of $\l^3$ as surfaces obtained by rotating a planar curve and then imposing that the mean curvature vanishes on the surface. In contrast, we revisit the rotational maximal surfaces of $\l^3$  as solutions of suitable Bj\"{o}rling problems when the core curve is a circle of $\l^3$. The discussion depends on the causal character of the rotational axis.  
 
 \begin{enumerate}
 \item Timelike axis. Suppose that the axis is $(0,0,1)$. A rotational maximal surface with axis $(0,0,1)$ is the   Bj\"{o}rling surface when the core curve is the circle $\alpha(t)=(\cos(t),\sin(t),0)$ and the normal vector field along $\alpha$ is $V(t)=\sinh(a) \n(t)+\cosh(a) \b(t)$, $a\in\r$, where  $$\n(t)=(-\cos(t),-\sin(t),0),\quad\quad \b(t)=(0,0,1).$$ If $a=0$ the Bj\"{o}rling surface  is the plane of equation $z=\mbox{constant}$. If $a\not=0$, then (\ref{bj}) gives  
 \begin{equation}\label{rott}
 X(u,v)=\left(
 \begin{array}{c} \cos (u) (\cosh (a) \sinh (v)+  \cosh (v))\\ 
  \sin (u)(\cosh (a) \sinh (v)+ \cosh (v))\\ 
  -v \sinh (a)
  \end{array}\right).
  \end{equation}
 The one-parametric group of rotations with axis $(0,0,1)$ is  
 $$ \left\{R_\theta=
\left(
\begin{array}{ccc}
\cos(\theta)  & -\sin(\theta)&0\\
\sin(\theta) & \cos(\theta) & 0 \\
0 & 0&1 \\
\end{array}
\right):\theta\in\r\right\}.$$
Then it is immediate that   $R_\theta\cdot X(u,v)=X(u+\theta,v)$, proving that $X(u,v)$ is rotational.  The isotropic curve is 
$$\psi(z)=(-\sin(z)-i\cosh(a)\cos(z), \cos(z)-i\cosh(a)\sin(z), i\sinh(a))$$
and the Weierstrass representation is
$$g(z)=-\tanh(a/2) e^{iz},\quad\omega=-\frac{i(1+e^a)^2}{2 e^{a}}e^{-iz}dz.$$ 
This surface is called the elliptic catenoid. 

\item  Spacelike axis. Suppose that the axis is $(1,0,0)$.  The rotational maximal surface is  the Bj\"{o}rling surface when the core curve is the spacelike hyperbola $\alpha(t)=(0,\sinh(t),\cosh(t))$ and $V(t)=\cosh(a) \n(t)+\sinh(a) \b(t)$, $a\in\r$, where  
$$\n(t)=(0,\sinh(t),\cosh(t)),\quad\quad \b(t)=(1,0,0).$$
Expression (\ref{bj}) gives 
\begin{equation}\label{rots}
X(u,v)=\left(\begin{array}{c}
 v \cosh (a)\\
 \sinh (u)( \sinh (a) \sin (v)+ \cos (v))\\
 \cosh (u)(\sinh (a)  \sin (v)+  \cos (v))
 \end{array}\right).
 \end{equation}
Since  the one-parametric group of rotations with axis $(1,0,0)$ is  
 $$ \left\{H_\theta=
\left(
\begin{array}{ccc}
 1  & 0&0\\
0 & \cosh(\theta) & \sinh(\theta) \\
0 & \sinh(\theta)& \cosh(\theta) \\
\end{array}
\right):\theta\in\r\right\},$$
 we find  that $H_\theta\cdot X(u,v)=X(u+\theta,v)$, proving that $X(u,v)$ is rotational.  This surface is called the hyperbolic catenoid.  The isotropic curve is 
 $$\psi(z)=(-i\cosh(a),\cosh(z)-i\sinh(a)\sinh(z),\sinh(z)-i\sinh(a)\cosh(z))$$
 and the Weierstrass representation is
 $$g(z)=\frac{i e^{a+z}+e^a-e^z-i}{e^{a+z}+i e^a+i e^z+1}\quad\omega=-\frac{ \left(e^{a+z}+i e^a+i e^z+1\right)^2}{4 e^{a+z}} dz.$$
 \item Lightlike axis. Suppose that the axis is $(1,0,1)$. A rotational maximal surface with axis    $(1,0,1)$  is the solution of the  Bj\"{o}rling problem  with $\alpha(t)=(-1+t^2/2,t,t^2/2)$ and $V(t)=\sinh(a) e_2(t)+\cosh(a) e_3(t)$, $a\in\r$, where $e_2(t)=\n(t)-\b(t)$ and $e_3(t)=\n(t)+\b(t)$ and 
$$\n(t)=\frac12\left(1,0,1\right),\quad\quad\b(t)=\frac12\left( t^2-1,2t, t^2+1\right).$$
Then the  Bj\"{o}rling surface given by (\ref{bj}) is 
\begin{equation}\label{lig2}
X(u,v)=\left(
\begin{array}{c} 
e^{-a}(  \frac{1}{6} v^3- \frac{1}{2} u^2v )+ \cosh (a)v+\frac{u^2}{2}-\frac{v^2}{2}\\
  u-e^{-a}uv\\
e^{-a}(  \frac{1}{6} v^3- \frac{1}{2} u^2v )+ \sinh (a)v+\frac{u^2}{2}-\frac{v^2}{2}
\end{array}\right)-\left(
\begin{array}{c} 
 1\\
 0\\
0
\end{array}\right).
\end{equation}
The one-parametric group of rigid motions with  axis $(1,0,1)$ is
$$ \left\{P_\theta=
\left(
\begin{array}{ccc}
 1-\frac{\theta^2}{2} & \theta & \frac{\theta^2}{2} \\
 -\theta & 1 & \theta \\
 -\frac{\theta^2}{2} & \theta & 1+\frac{\theta^2}{2}  \\
\end{array}
\right):\theta\in\r\right\},$$
and it is immediate  that   $P_\theta \cdot X(u,v)=X(u+\theta,v)$ for all $\theta$, proving that $M$ is rotational. The isotropic curve of  the surface $X$ is
$$\psi(z)=(z+\frac{i}{2}(e^{-a}z^2-2\cosh(a)), 1+ie^{-a}z, z+\frac{i}{2}(e^{-a} z^2-2\sinh(a)))$$
and the Weierstrass representation is 
$$g(z)=\frac{e^a+iz-1}{e^a+i z+1},\quad\omega=-\frac{ i  \left(e^a+i z+1\right)^2}{2 e^a}dz.$$
This surface is called the parabolic catenoid. 
   \end{enumerate}

Up to dilations and rigid motions, in each one of the three above cases, and independently of the value of the parameter $a\in\r$ in the parametrizations (\ref{rott}), (\ref{rots}) and (\ref{lig2}), the surface  is the unique rotational maximal surface of $\l^3$. For example, and when the axis is timelike, if we multiply $\omega$ by a real number, the surface changes by a dilation.   Thus we suppose $\omega=ie^{-iz}dz$. With the change $-\tanh(a/2) e^{iz}\rightarrow z$, the Weierstrass representation is $g(z)=z$ and $\omega=-\tanh(a/2) dz/z^2$. After  a dilation again, we conclude $g(z)=z$ and $\omega=dz/z^2$, where it does not appear the initial parameter $a$. In the literature and following Kobayashi (\cite{ko}), the above three catenoids are also called as   catenoid of first kind, catenoid of second kind for the surface  and the Enneper  surface of second kind  respectively.

%%%%%%%%%%%%%%%%%%%%%%%%%%%%%%%%%%%%%%%%%%
\section{Duality of rotational minimal and maximal  surfaces }\label{s-dual}
%%%%%%%%%%%%%%%%%%%%%%%%%%%%%%%%%%%%%%%%%%

In this section  we address  with Problem 2  and we ask if the dual surface of a rotational minimal (maximal) surface is, indeed, rotational.  We begin calculating the dual surfaces of   rotational maximal surfaces of $\l^3$, where the classification depends on the causal character of the rotational axis. In the next computations, we will suppose that the rotational axis is $(0,0,1)$ (timelike),  $(1,0,0)$ (spacelike) and $(1,0,1)$ (lightlike) and we use the isotropic curves and the Weierstrass representations obtained in Sect. \ref{appendix}. 

\begin{enumerate}
\item  Timelike axis.  The Weierstrass representation of the elliptic catenoid $C$ is   $g(z)=z$ and $\omega=dz/z^2$ defined in $M=\c-\{0\}$.  Then the Weierstrass representation of $C^\sharp$ is $g^\sharp(z)=-iz$ and $\omega^\sharp=idz/z^2$. Up to the automorphism $-iz\rightarrow z$, we have $g^\sharp(z)=z$ and $\omega^\sharp =dz/z^2$, which is the Weierstrass representation of the Euclidean catenoid of axis $(0,0,1)$: compare with Rem. \ref{remark2}.

\item Spacelike axis. The Weierstrass representation of the hyperbolic catenoid $C$ is  
$$g(z)=i\frac{e^z-1}{e^z+1},\quad\omega=-i\frac{(e^z+1)^2}{2e^z}dz$$
defined  in $M=\c$.  Then 
\begin{equation}\label{weispacelike}
g^\sharp(z)=\frac{e^z-1}{e^z+1},\quad\omega^\sharp=\frac{(e^z+1)^2}{2e^z}dz.
\end{equation}
An integration of (\ref{weierstrass}) gives the parametrization of $C^\sharp$:   
 $$X^\sharp(u,v)=\left( u,  -\cosh(u)\sin(v),  \cos(v)\cosh(u) \right)-(0,0,1).$$
This surface is the Euclidean catenoid with respect to  the axis of equation $y=0, z=-1$.  
\item Lightlike axis.  By taking $a=0$ in (\ref{lig2}), the Weierstrass representation of the parabolic catenoid $C$ is $$g(z)=\frac{z}{z-2i},\quad\omega=\frac{ i  \left(z-2i\right)^2}{2 }dz.$$
Then   $C^\sharp$  is given by 
\begin{equation}\label{ww2}
g^\sharp(z)=-\frac{iz}{z-2i},\quad \omega^\sharp =-\frac{ \left( z-2i \right)^2}{2}dz
\end{equation}
 and its parametrization  by means of (\ref{weierstrass}) is  
 $$X^\sharp(u,v)=\left(\frac{6( u-uv)-u^3+3  uv^2}{6},  \frac{v^2-u^2-2v}{2},  \frac{3(u^2-v^2-u^2 v)+v^3}{6}   \right).$$
 We prove  that $C^\sharp$ is the Enneper surface. First, the automorphism $z-2i\rightarrow z$ and a dilation of $\e^3$ changes the Weierstrass representation (\ref{ww2}) into
 $$g^\sharp(z)=-i\frac{z+2i}{z},\quad\omega^\sharp=-z^2dz.$$
 The isotropic curve $\psi^\sharp$ is now
 $$\psi^\sharp=\left(\frac{i(2z^2+2iz-1}{2},\frac{2iz-1}{2}, z^2+iz \right),$$
 which has not real periods. The metric of $C^\sharp$ is 
 $$ds^\sharp=\frac12|\omega^\sharp|(1+|g^\sharp(z)|^2)=\left(\frac{|z|^2}{2}+\frac{|z+2i|^2}{2}\right)|dz|.$$ 
 Because $ds^\sharp\geq|z|^2/2|dz|$, then it is immediate that $ds^\sharp$  is complete. As the degree of $g^\sharp$ is $1$,  $C^\sharp$ has total finite curvature equal to $-4\pi$. Since the surface is not the catenoid, it is the Enneper surface by the classification of Osserman (\cite{oss}).
\end{enumerate}
We explicit the above calculations.

\begin{proposition}\label{t1} Let $M$ be a catenoid of $\l^3$ with axis $(0,0,1)$ or $(1,0,0)$. Then its dual surface is the  Euclidean catenoid with the same rotational axis.  The dual surface of the parabolic catenoid of axis   $(1,0,1)$  is   the Enneper surface.
\end{proposition}
 
From this result, and reversing the dual process,  the dual surfaces of the Euclidean catenoids of axis $(0,0,1)$ and $(1,0,0)$ are two non-congruent maximal surfaces in contrast to the existence of  a deformation by congruences between the two initial Euclidean catenoids. Therefore we have a first example to give a negative answer to  Problem 1. On the other hand, it also provides some answers to Problem 2: the dual surfaces of the elliptic and the hyperbolic catenoid are rotational, but  the dual surface of the parabolic catenoid is not rotational. 

We study the dual surfaces  of the Euclidean catenoid under this deformation. We describe explicitly our setting. Consider  $C$  the Euclidean catenoid with axis $(0,0,1)$ and let $\{G_t:t\in [0,\pi/2]\}$ be the one-parametric group of rotations about $(0,1,0)$.  Then   $G_0=\mbox{id}$ and 
$$ G_t=\left(
\begin{array}{ccc}
 \cos (t) & 0 & \sin (t) \\
 0 & 1 & 0 \\
 -\sin (t) & 0 & \cos (t) \\
\end{array}
\right).$$
Thus $G_t$ is a rotation about the orthogonal line  to the axis of $C$ passing through the origin that transforms the axis   $(0,0,1)$ for $t=0$ into $(1,0,0)$ when $t=\pi/2$.   We rotate $C$    by means of $G_t$.  Along this rotation of $C$ we obtain a  one-parametric family of catenoids $\{C_t:=G_t(C):t\in [0,\pi/2]\}$ which, by Prop. \ref{t1}, we know that the dual surfaces of   the first and the last catenoid, namely, $C_0^\flat$ and $C_{\pi/2}^\flat$ are the elliptic catenoid and the hyperbolic catenoid, respectively. In particular,  $C_0\simeq C_{\pi/2}$ but   $C_0^\flat\not\simeq C_{\pi/2}^\flat$. In this context we ask which are the dual surfaces  $C_t^\flat$ during this process.

We compute the isotropic curves of $C_t$ and $C_t^\flat$ and the corresponding Weierstrass representation. This can done by solving the Bj\"{o}rling problem where the core curve is 
$\alpha(s)=G_t(\cos(s),\sin(s),0)$ and $V(s)$ is the  unit normal vector of $\alpha$, that is, $V(s)=-\alpha''(s)$. The isotropic curves $\phi$ of $C_t$  and $\psi=\phi^\flat$ of $C_t^\flat$ are
$$\phi(z)=\left( -\cos (t) \sin (z)-i \sin (t),\cos (z),\sin (t) \sin (z)-i \cos (t) \right),$$
$$\psi(z)=\left(-\sin (t)+i \cos (t) \sin (z),-i \cos (z),\sin (t) \sin (z)-i \cos (t)\right).$$
We can obtain an explicit parametrization of   $C_t^\flat$ by integrating (\ref{weierstrass}), obtaining
\begin{equation}\label{xtb}
X_t^\flat(u,v)=\left(\begin{array}{c}
-u \sin (t)-\cos (t) \sin (u) \sinh (v)\\ 
\cos (u) \sinh (v)\\
-\sin (t) \cos (u) \cosh (v)+v \cos (t)+\sin (t)
\end{array}\right).\end{equation}

A simple computation yields   the Weierstrass representation of  $C_t^\flat$: 
$$g^\flat(z)=\frac{e^{i (t+z)}+i e^{i t}+e^{i z}-i}{-e^{i (t+z)}-i e^{i t}+e^{i z}-i},\quad\omega^\flat=\frac{ \left(-e^{i (t+z)}-i e^{i t}+e^{i z}-i\right)^2}{4 e^{i (t+z)} }dz,$$
respectively. With the change $z\rightarrow iz$ and after some transformations, we obtain
$$g^\flat(z)=i\frac{\cos\frac{t}{2}e^z-\sin\frac{t}{2}}{\sin\frac{t}{2}e^z+\cos\frac{t}{2}},\quad\omega^\flat=-\frac{ \left(\sin\frac{t}{2}e^z+\cos\frac{t}{2}\right)^2}{i e^z}dz.$$
Up to a multiplication and the division by $i$ in $g^\flat$ and $\omega^\flat$ (see Rem. \ref{remark2}),  we have
\begin{equation}\label{bol}
g^\flat(z)=\frac{\cos\frac{t}{2}e^z-\sin\frac{t}{2}}{\sin\frac{t}{2}e^z+\cos\frac{t}{2}},\quad\omega^\flat=\frac{ \left(\sin\frac{t}{2}e^z+\cos\frac{t}{2}\right)^2}{e^z}dz.
\end{equation}
The maximal surfaces with this  Weierstrass representation    are the Bonnet maximal surfaces in $\l^3$ obtained by Leite in \cite{lei}: compare  the Weierstrass data (\ref{bol}) with the  analogous surfaces in Euclidean space which will appear in (\ref{boe}) and the parametrization $X_t^\flat$ in (\ref{xtb}) with the parametrization of a Bonnet minimal surface of $\e^3$ described in (\ref{xtb2}). With    the change $e^{z}\rightarrow z$, the isotropic curve of $M^\flat$ is  
$$\psi=\left(\frac{z^2+1}{2 z^2},\frac{i\cos(t)(1-z^2) -2 iz \sin (t)}{2 z^2},\frac{2 z \cos (t)+\sin (t)(1-z^2)}{2 z^2}\right)$$
defined in  $M=\c-\{0\}$. We point out that the Weierstrass representation of $M^\flat$ has real periods for $\psi_2$, except at $t=0,\pi/2$, which means that the surface is singly periodic.

\begin{theorem} Let  $\{C_t:t\in [0,\pi/2]\}$ be the one-parametric family of Euclidean catenoids where the parameter $t$ indicates the angle that makes the rotation axis  with the $z$-axis. Then the dual surfaces $C_t^\flat$ belong to the  family of Bonnet surfaces in $\l^3$ starting from the elliptic catenoid for $t=0$ until the hyperbolic catenoid for $t=\pi/2$.
\end{theorem}

Thus $C_t\simeq C_s$ but $C_t^\flat\not\simeq C_s^\flat$ for any $s,t\in [0,\pi/2]$. From the Lorentzian viewpoint, the axis of $C_t$ is timelike if $t\in [0,\pi/4)$, spacelike if $t\in (\pi/4,\pi/2]$ and lightlike if $t=\pi/4$.  Then we observe that when the axis $L$ is $(0,0,1)$ and $(1,0,0)$, the dual surface of the Euclidean catenoid of axis $L$ is a catenoid in $\l^3$ with the same axis. However, when $t=\pi/4$, the surface $C_{\pi/4}^\flat$ is not   the parabolic catenoid, but a surface in the Bonnet family of maximal surfaces.
  In the next pictures, we show the dual surfaces of the catenoids $C_t$ when they are close to the  elliptic catenoid (Figure \ref{fig1}) and close  to the hyperbolic catenoid (Figure \ref{fig2}).
\begin{figure}[hbtp]
%\begin{center}
\includegraphics[width=.3\textwidth]{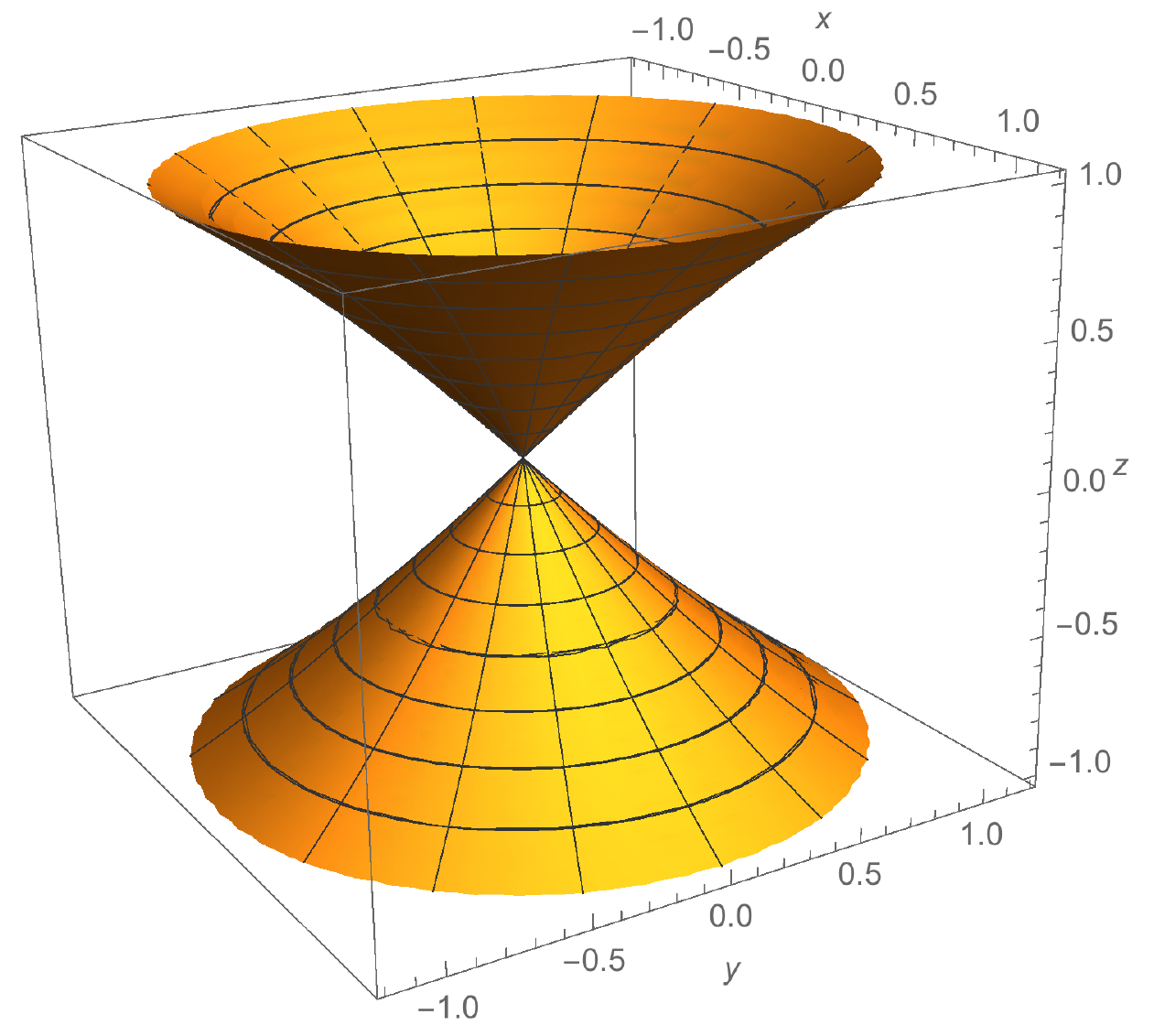} \includegraphics[width=.3\textwidth]{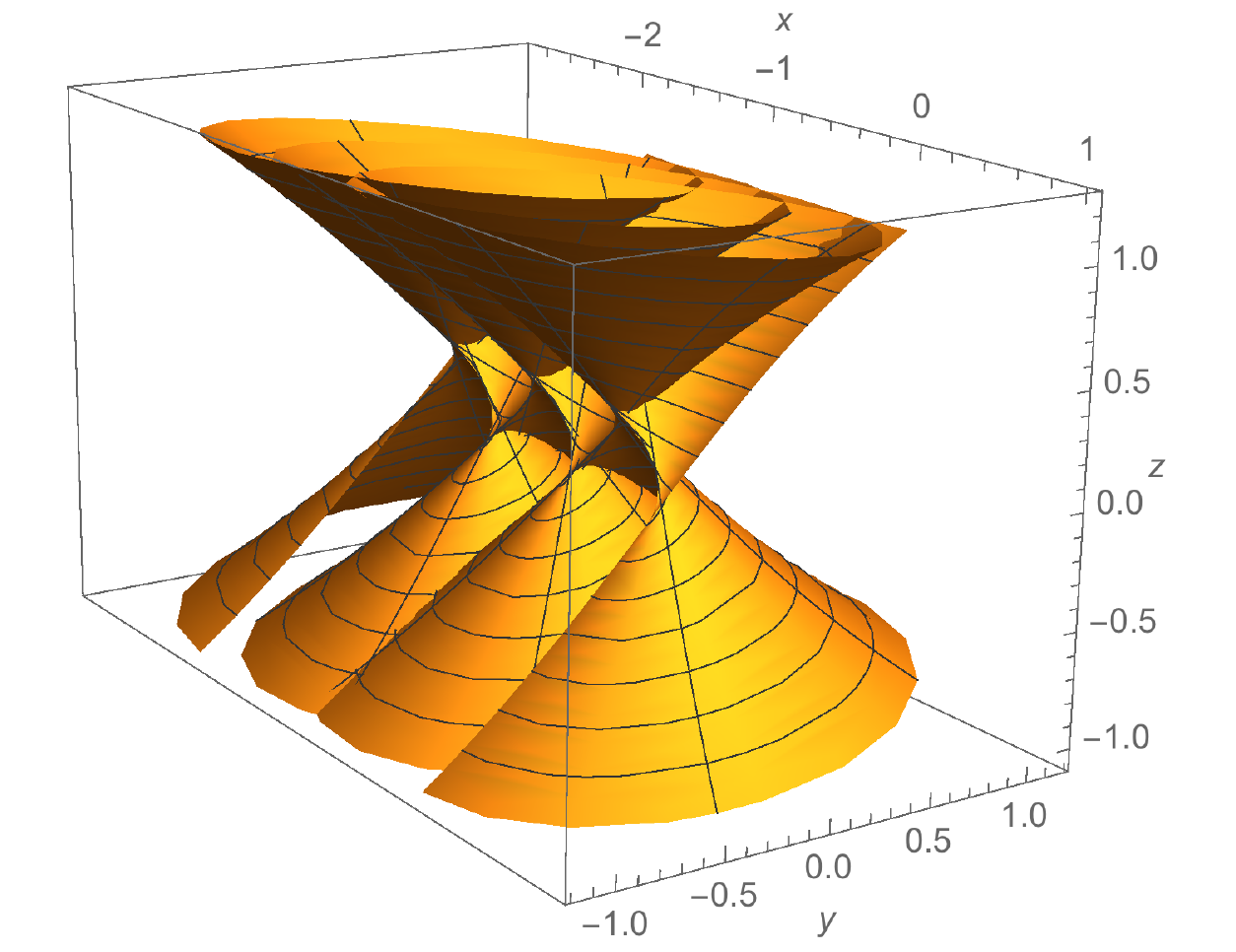}\includegraphics[width=.3\textwidth]{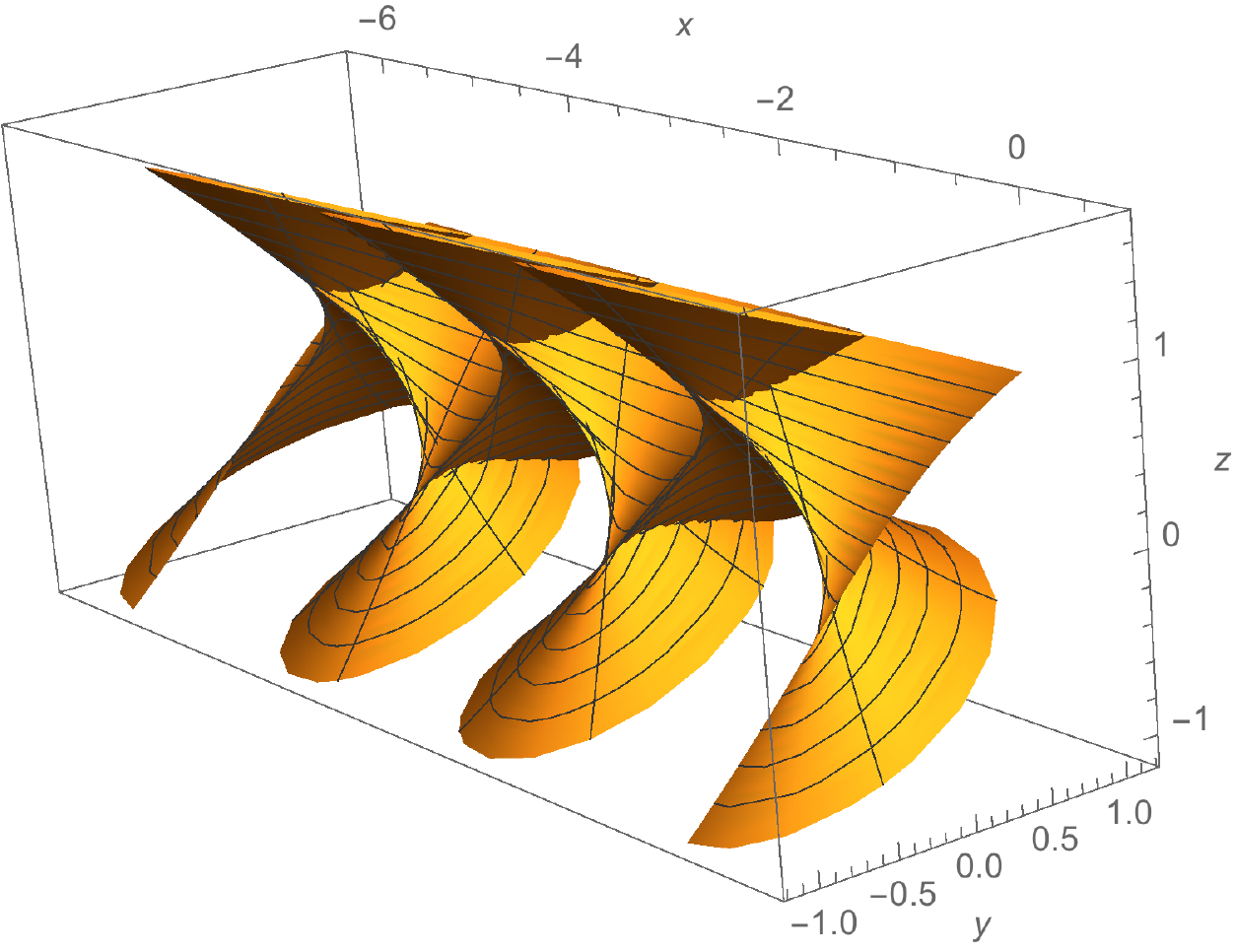}
%\end{center}
\caption{The surfaces $C_t^\flat$ for $t=0$ {\it (left)}, $t=0.1$ {\it (middle)} and $t=0.3$ {\it(right)}}\label{fig1}
\end{figure}
\begin{figure}[hbtp]
%\begin{center}
\includegraphics[width=.3\textwidth]{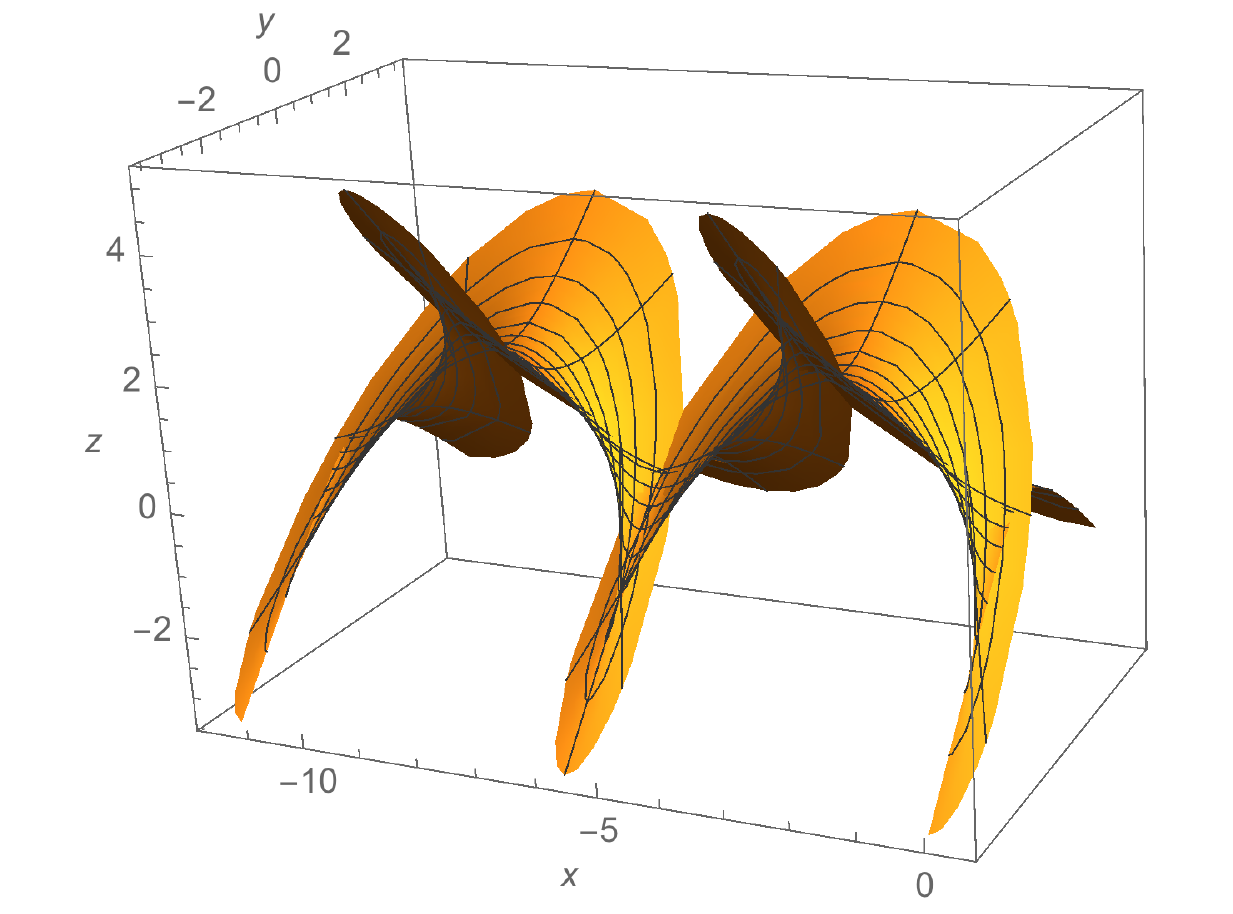} \includegraphics[width=.3\textwidth]{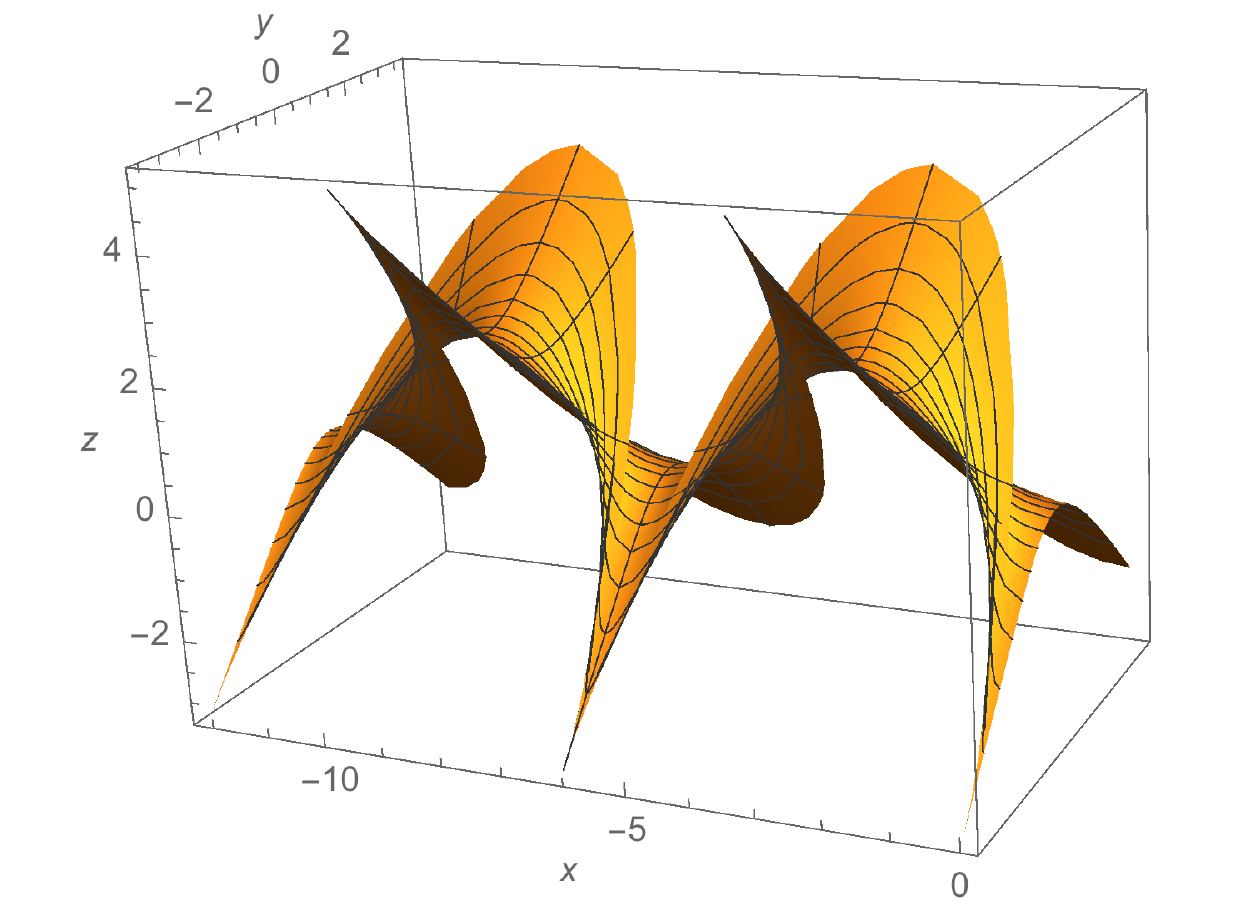}\includegraphics[width=.3\textwidth]{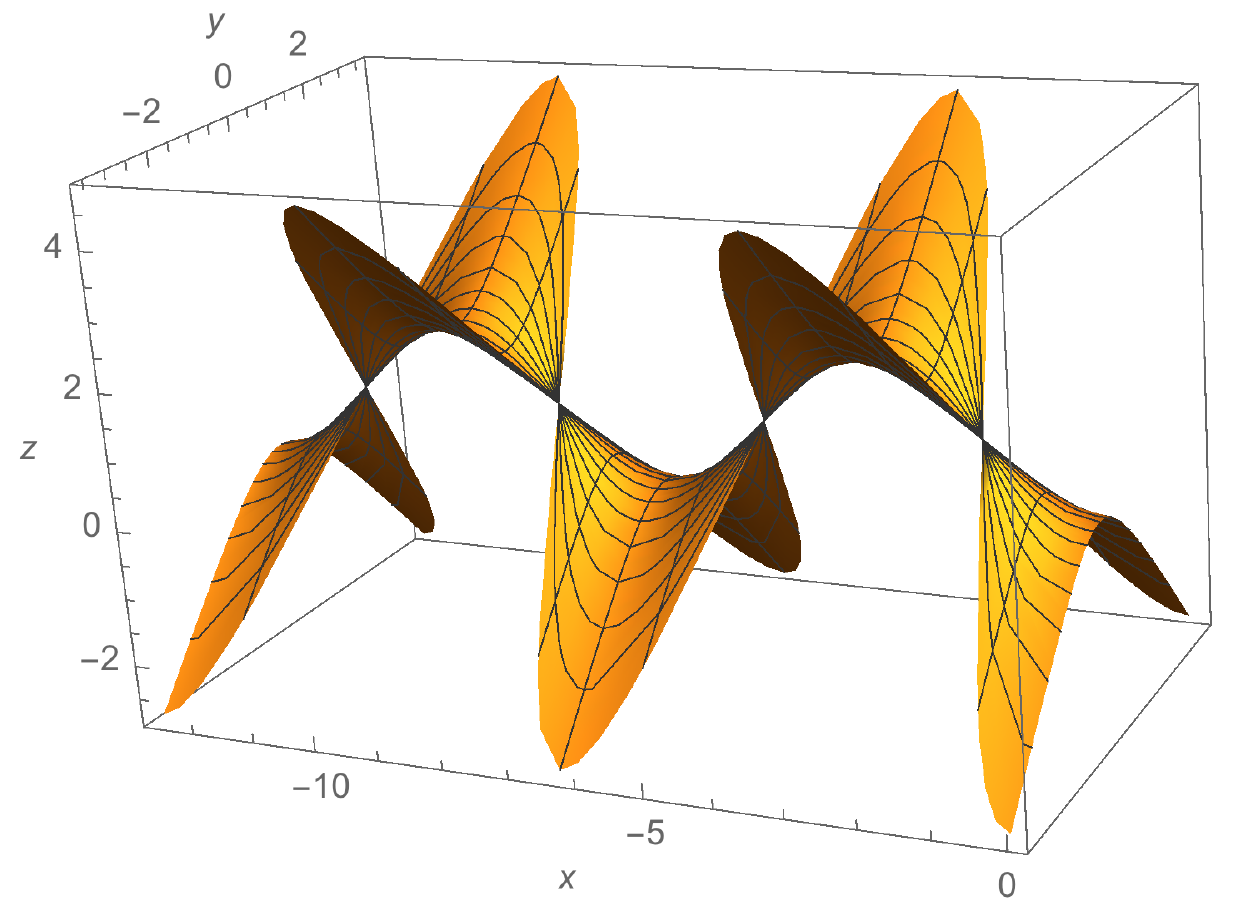}
%\end{center}
\caption{The surfaces $C_t^\flat$ for $t=1.1$ {\it (left)}, $t=1.3$ {\it (middle)} and $t=\pi/2$ {\it(right)}}\label{fig2}
\end{figure}

Proposition \ref{pr1} allows to consider the dual surfaces of a helicoid.  In both ambient spaces, minimal and maximal  ruled surfaces are called helicoids.   It was  proved in \cite{mp} that   an associate surface of a catenoid of $\l^3$ with axis $L$ is characterized to be  a maximal surface invariant by a one-parameter group of helicoidal motions.   Since the adjoint surface of the Euclidean catenoid is the helicoid (with the same axis) and the adjoint of the Enneper surface coincides with itself up to a reparametrization and a rotation,  Propositions \ref{pr1} and   \ref{t1} conclude:
 
 \begin{corollary}\label{co1} The dual surface of a helicoid of $\l^3$ with axis $(0,0,1)$ or $(1,0,0)$ is a helicoidal surface of $\e^3$ with the same axis. The dual surface of the helicoid of axis $(1,0,1)$   is the Enneper surface.   \end{corollary}
 
%%%%%%%%%%%%%%%%%%%%
\section{Deformations of the Lorentzian catenoid and  duality }\label{s-dual2}
%%%%%%%%%%%%%%%%%%%%%%%%%%%%%%

In this section  we  deform a   catenoid $C\subset\l^3$   by  a one-parameter group of rotations of $\l^3$ and we ask which are their dual surfaces. Now  we have three types of catenoid and   we can deform under the three types of rotations of $\l^3$.  We separate in cases depending on the  type of the  catenoid $C$. On the other hand, by Prop. \ref{pr1}, the rotations $R_t$ about the axis $(0,0,1)$  satisfy     $R_t(M)^\sharp\simeq M^\sharp$, and thus we discard this type of rotations because they do not provide new minimal surfaces. As a conclusion,    the rotations that we consider are the hyperbolic  rotations $H_t$ about $(1,0,0)$ and the parabolic rotations $P_t$ about $(1,0,1)$, see Sect. \ref{appendix}.   We also do not use rotations with the same axis of the given catenoid because they do not change the initial catenoid. 

\subsection{Elliptic catenoid}
Consider an elliptic catenoid $C$ with rotation axis  $(0,0,1)$.  Here $C$ is determined by the Weierstrass representation $g(z)=e^z$ and $\omega=e^{-z}dz$ in $M=\c$  and the isotropic curve is 
$\psi(z)=(\cosh (z), -i\sinh (z),1)$.   
\begin{enumerate}
\item Hyperbolic rotations. We consider  the family of surfaces  $\{H_t(C)^\sharp:t\in\r\}$.  Then the isotropic curve of $H_t(C)$ is   
$$H_t(\psi)=\left(\cosh(z),-i  \sinh(z) \cosh (t)+\sinh (t),\cosh (t) -i   \sinh(z) \sinh (t)\right).$$
The isotropic curve of the dual surface $H_t(C)^\sharp$ is
$$H_t(\psi)^\sharp=\left( i \cosh(z),\sinh(z) \cosh (t)+ i \sinh (t),\cosh (t)-i \sinh(z) \sinh (t)\right)$$
and its  Weierstrass representation   is
\begin{equation}\label{boe}
g^\sharp(z)=-i\frac{e^z \cosh  \frac{t}{2} +i\sinh  \frac{t}{2} }{-ie^z \sinh  \frac{t}{2}+\cosh  \frac{t}{2}  },\quad \omega^\sharp=i e^{-z} \left(\cosh  \frac{t}{2} -ie^z \sinh  \frac{t}{2} \right)^2 dz.
\end{equation}
We can integrate explicitly the parametrization of the surface by means of $H_t(\psi)^\sharp$, obtaining
\begin{equation}\label{xtb2}
Y_t(u,v)=\left(\begin{array}{c}
 -\cosh (u) \sin (v),\\
 \cosh (t) \cosh (u) \cos (v) -v\sinh(t)-\cosh (t),\\
 \sin(v)\sinh(u) \sinh (t) +u \cosh (t) 
\end{array}\right).
\end{equation}
See Figure \ref{fig3}, left. As it is known by Prop. \ref{t1},  for $t=0$ we obtain the Euclidean catenoid of axis $(0,0,1)$. For $t\not=0$, the parametrization $Y_t$ in (\ref{xtb2}) coincides with the parametrizatios of  the family of minimal surfaces discovered by Bonnet in \cite{bo}, see \cite[\S 175]{ni}. Besides the plane, the catenoid and the Enneper surface, the Bonnet surfaces   are the only nonplanar minimal surfaces with planar lines of curvature (see   \cite{co}  from a different point of view). 

\begin{remark}\label{rm}  The Weierstrass representation of the Bonnet minimal surfaces in Euclidean space $\e^3$ is, up to congruences and dilations,   $g(z)=e^z+\lambda$ and $\omega=e^{-z}dz$ defined in $M=\c$, where $\lambda\in (0,\infty)$. We denote this surface  by $\mathcal{B}(\lambda)$. After the change $e^z\rightarrow z$ , it is not difficult to see that if $\lambda\not=0$, the isotropic curve $\phi$ has real periods in the second coordinate $\phi_2$ which means that the surface is singly periodic.  The Weierstrass representation of the Bonnet maximal surfaces coincide with the Euclidean case (\cite{al}).
\end{remark}

\item Parabolic rotations. We consider the family of surfaces $\{P_t(C)^\sharp:t\in\r\}$ obtaining that the isotropic curve is
$$P_t(\psi)^\sharp=\left(\begin{array}{c}
 i \left(t (t-2 i \sinh (z))-\left(t^2-2\right) \cosh (z)\right)/2\\
 it(1-\cosh (z))+\sinh (z)\\
   \left(t^2-t (t \cosh (z)+2 i \sinh (z))+2\right)/2 \end{array}\right)$$
and its  Weierstrass representation   is
\begin{equation}\label{ell2} 
g^\sharp(z)=-i\frac{ (t+2 i) e^z-t}{ t e^z+2 i-t},\quad \omega^\sharp=-i\frac{\left(t e^z+2 i-t\right)^2} {4  e^z}  dz.
\end{equation}
We find an explicit parametrization of the surface by means of (\ref{weierstrass}), yielding
\begin{equation}\label{ztu}
Z_t(u,v)=\left(\begin{array}{c}
  t^2 \cosh (u) \sin (v)/2- t^2 v/2+t \cosh (u) \cos (v)-t-\cosh (u) \sin (v)\\
  \cosh (u) (t\sin (v)+\cos(v))-t v -1\\
- t^2 \sinh (u) \cos (v)/2+ t^2 u/2+t \sinh (u) \sin (v)+u\end{array}
\right).
\end{equation}
See Figure \ref{fig3}, right. We prove that the surface $P_t(C)^\sharp$ is a Bonnet surface  after a suitable Goursat transformation. Here we recall  a Goursat transformation of a minimal surface. If $\phi$ is the isotropic curve of a minimal surface $M\subset \e^3$, a Goursat transformation of $M$ is the minimal surface whose isotropic curve is $A\phi$, where $A$ is an element of the complex rotation group $O(3,\c)$ (\cite{go}). In terms of the conformal parameter, a Goursat transformation is characterized by a change of the Gauss map  under a M\"{o}bius transformation $T \in \mbox{Aut}(\c\cup\{\infty\})$ that preserves the Hopf differential (\cite[p. 206]{udo}). In particular, the Weierstrass representation    $(g,\omega)$ of $M$ changes into $\{T(g),\omega/T'(g)\}$ of $T(M)$.

Returning with the surface $Z_t(u,v)$ described in (\ref{ztu}),  we are able  to find a M\"{o}bius transformation $T(z)=(az+b)/(cz+d)$ such that $T(e^z+\lambda)=g^\sharp(z)$, where $\lambda\in (0,\infty)$ is the Bonnet parameter. For this, and in view of the Gauss map $g^\sharp(z)$,  we fix $a=2-it$ and $c=t$. Because $\lambda$ is a positive real number, it is not difficult to find that $b$ and $d$ are given by 
$$b=\mu -\frac{1}{2} i (\mu -2) t,\ d=\frac{1}{2} (\mu -2) t+2 i$$
and $\mu<0$ is a real parameter. The Bonnet  parameter   is 
$\lambda=-\mu/2$. This proves definitely that if $\mathcal{B}(\lambda)$ is the Bonnet minimal surface for the above value of $\lambda$, then $T(\mathcal{B}(\lambda))=P_t(C)^\sharp$.

In (\ref{ell2}) we do the change $e^z\rightarrow z$ and after a $\pi/2$-rotation about the $z$-axis and and reflection about the $xy$-plane, the Weierstrass representation of $P_t(C)^\sharp$ is 
$$g^\sharp(z)= \frac{ (t+2 i) z-t}{ t z+2 i-t},\quad \omega^\sharp= \frac{\left(t z+2 i-t\right)^2} {4z^2}  dz.$$
The isotropic curve is
$$\phi=\left(\begin{array}{c}
\dfrac{-1-it+2itz+(1-it)z^2}{2 z^2}\\
\dfrac{-2i+2t+it^2-2it^2 z+i(-2+2it+t^2)z^2}{4 z^2}\\
\dfrac{t^2-2it-(2t^2+4)z+(t^2+2it)z^2}{4 z^2}
\end{array}\right).$$
Thus the surface $P_t(C)^\sharp$, $t\not=0$, has real periods along the only non-trivial homological curve of $M$ and this says that $P_t(C)^\sharp$ is singly periodic. 
\end{enumerate}
We summarize the result in the next theorem:

\begin{figure}[hbtp]
%\begin{center}
\includegraphics[width=.48\textwidth]{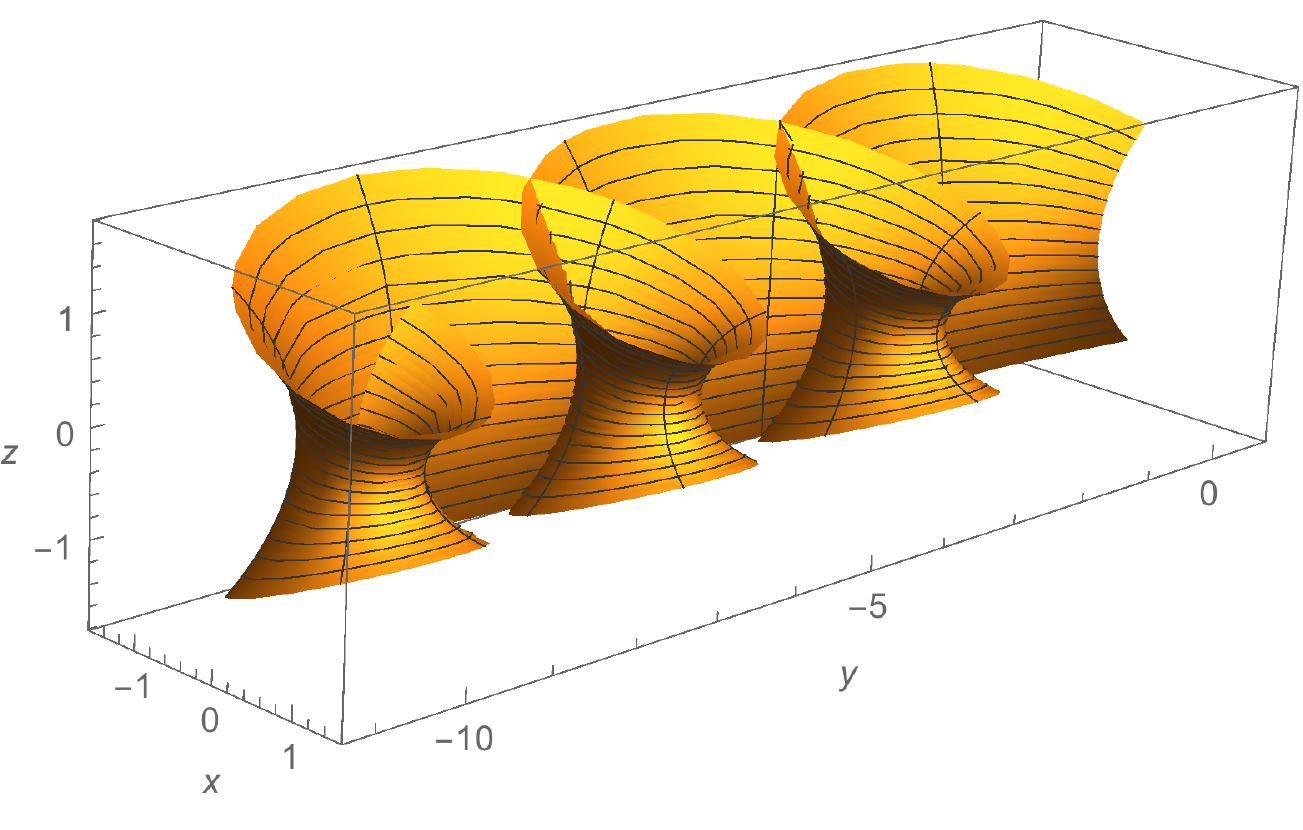} \includegraphics[width=.48\textwidth]{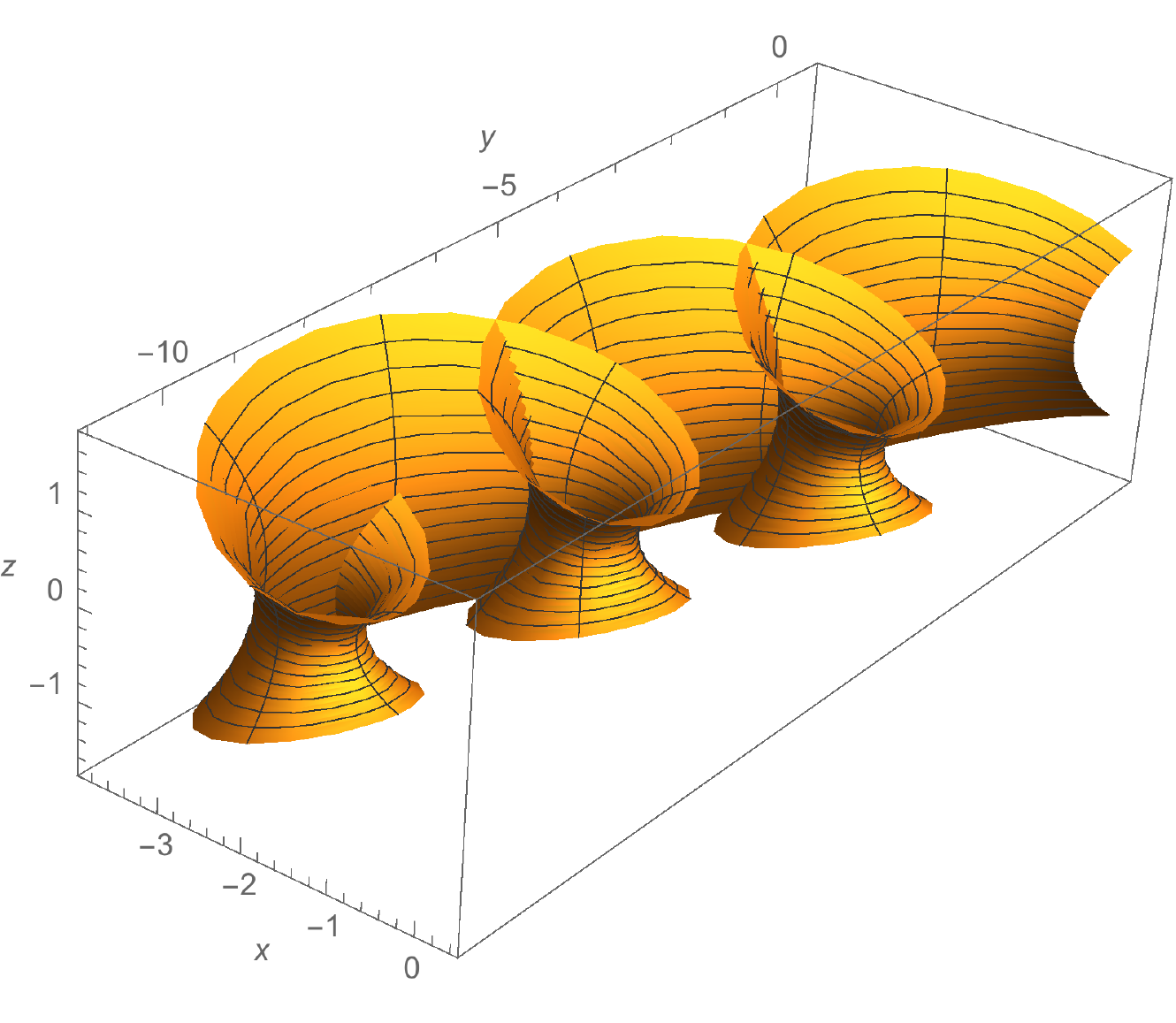}
%\end{center}
\caption{Theorem \ref{t51}: the surface $Y_t$ for $t=1$ ({\it left}) and the surface $Z_t$ for $t=1/2$ ({\it right})}\label{fig3}
\end{figure}
 
\begin{theorem}\label{t51} Consider $C\subset\l^3$ the elliptic catenoid of axis $(0,0,1)$. 
The dual surface of $C$ by the hyperbolic rotation $H_t$ about $(1,0,0)$  is a  Bonnet minimal surface. The dual surface of $C$ by a parabolic rotation $P_t$ about $(1,0,1)$ is  a Goursat transformation of a Bonnet minimal surface.
\end{theorem}

%%%%%%%%%%%%%%%%%%
\subsection{Hyperbolic catenoid}
Let $C$ be the hyperbolic catenoid  of axis $(1,0,0)$ whose Weierstrass representation is $g(z)= i\frac{e^z-1}{e^z+1}$ and $\omega=-i\frac{(e^z+1)^2}{2e^z}dz$ defined in $M=\c$.   If $\psi$ is the isotropic curve of $C$, then the isotropic curve of $P_t(C)^\sharp$ is 
$$P_t(\psi)^\sharp= \frac{1}{2}\left(\begin{array}{c}
  2-t^2+i t (t \sinh (z)+2 \cosh (z)) \\
2 i (\cosh (z)+t (\sinh (z)+i))\\
 \left(t^2+2\right) \sinh (z)+i t^2+2 t \cosh (z) 
  \end{array}\right)$$
and the Weierstrass representation  is
$$g^\sharp(z)= \frac{(1+t+i) e^z-1-i+it}{(1+i+it) e^z+1+i-t},\quad\omega^\sharp= \frac{\left((1-i+t) e^z+1-i+it\right)^2}{4i e^z}dz.$$
By integrating $P_t(\psi)^\sharp$, the parametrization of $P_t(C)^\sharp$ is 
$$W_t(u,v)=\left(\begin{array}{c}
- t^2 u/2-  t \sin (v) (t \sinh (u)+2 \cosh (u))/2+u\\
-\sin (v) (t \sinh (u)+\cosh (u))-t u\\
  \left(\left(t^2+2\right) \cosh (u) \cos (v)-(v+1)t^2 +2 t \sinh (u) \cos (v)-2\right)/2
\end{array}\right).$$
We  prove that this surface is the Goursat transformation of a Bonnet minimal surface. The computations are similar as in Th. \ref{t51}. In view of $g^\sharp(z)$, we choose 
$$T(z)=\frac{az+b}{cz+d},\ \ a=1+t+i,\ c= 1+i+it.$$
Then let  
$$b=\mu+\frac{ \left(t^2+\mu\right)}{t+1}i,\ d=\frac{2+\mu-t^2}{1+t}+(2+\mu)i$$
and $\mu$ is a real parameter. The parameter $\lambda$ of the Bonnet minimal surface  is 
$\lambda=-(1+\mu)/(1+t)$ and the Weierstrass representation of the minimal surface $T(\mathcal{B}(\lambda))$ is $(g^\sharp(z),\omega^\sharp)$.  This proves that $T({\mathcal B}(\lambda))=P_t(C)^\sharp$.

With the change $e^z\rightarrow z$, and up to a $\pi/2$-rotation about the $z$-axis, dilations  and a reflection with respect to the $xy$-plane, we have 
$$g^\sharp(z)=\frac{(t+1+i) z-1-i+it}{(1+t-i) z+1-i+it},\quad\omega^\sharp=\frac{ ((1+t-i) z+1-i+it)^2}{4z^2}dz$$
and
 $$P_t(\psi)^\sharp=\left(\begin{array}{c}
\dfrac{-i+it+2tz-(i+it)z^2}{2 z^2}\\
\dfrac{it(2-t)+(4-2t^2)z+it(t+2)z^2}{4 z^2}\\
\dfrac{-2+2t-t^2+2it^2z+(2+2t+t^2)z^2}{4 z^2}\end{array}\right).$$
Let us observe that $P_t(\psi)^\sharp$ is defined in $\c-\{0\}$ and has real periods in the third coordinate. This means that the surface is singly periodic.
As a conclusion:

\begin{theorem}\label{t52} 
Consider $C$ the hyperbolic catenoid of axis $(1,0,0)$. 
The dual surface of $C$ by a parabolic rotation $P_t$   about $(1,0,1)$ is   a Goursat transformation of a Bonnet minimal surface.
\end{theorem}

%%%%%%%%%%%%%%%%%%
\subsection{Parabolic catenoid}
%%%%%%%%%%%%%%%%%%%%%%%%%%%%%%%%%

We consider the parabolic catenoid $C$ with axis $(1,0,1)$ and whose Weierstrass representation is $g(z)=z/(z-2i)$ and $\omega=i(z-2i)^2/2 dz$ defined in $M=\c$.  The isotropic curve of $C$ is 
$\psi(z)=\left(i z^2/2+z-i,1+i z,z+ i z^2/2\right)$. We deform $C$ by the   hyperbolic rotations $H_t$ about the axis $(1,0,0)$ and we compute  the  dual surfaces of $H_t(C)$. Then the isotropic curve of $H_t(C)^\sharp$ is    
\begin{equation}\label{hts}
H_t(\psi)^\sharp=\left(\begin{array}{c}
1+iz-z^2/2\\
i\cosh t+(-\cosh t+i\sinh t)z- \sinh{t}\ z^2/2\\
\sinh t+(\cosh t+i\sinh t) z+ i\cosh t\ z^2/2 
\end{array}
\right).
\end{equation}
The Weierstrass representation is 
$$g^\sharp(z)=\frac{(\sinh\frac{t}{2}-i\cosh\frac{t}{2} )z-2i\sinh\frac{t}{2}}{(\cosh\frac{t}{2}-i\sinh\frac{t}{2})z-2i\cosh\frac{t}{2}},\quad \omega^\sharp= 
-\frac{1}{2} \left((\cosh\frac{t}{2}-i\sinh\frac{t}{2})z-2i\cosh\frac{t}{2} \right)^2dz.$$
The parametrization of $H_t(C)^\sharp$ is 

$$X_t(u,v)=\frac{1}{6}\left(\begin{array}{c}
- u \left(u^2-3v^2+6v-6\right)\\
  -u \sinh (t) \left(u^2-3v^2+6v\right)-3 \cosh (t) \left(u^2-v^2+2v\right) \\
   \cosh (t) \left(-3 u^2 v+3 u^2+v^3-3 v^2\right)-(6uv-6u) \sinh (t).
\end{array}\right).$$
By (\ref{hts}), the isotropic curve $H_t(\psi)^\sharp$ has not real periods. We show that $H_t(C)^\sharp$  is a complete surface. Then it is not difficult to see that the induced metric $ds^\sharp$ on $H_t(C)^\sharp$ satisfies
$$ds^\sharp=\frac12(|\omega^\sharp|+|\omega^\sharp||g^\sharp|^2)\geq \frac12|\omega^\sharp|
\geq (a|z|^2+b|z|+c)|dz|$$
for certain numbers $a,b,c\in\r$, $a>0$, related with $\cosh(t/2)$ and $\sinh(t/2)$.  Then it is immediate that if $\gamma$ is a divergent path on $M=\c$, that is, if $\gamma$ is a path on $\c$ that has $\infty$ as an end point, then the length of $\gamma$ is $\infty$. Therefore  $H_t(C)^\sharp$ is a complete surface and because the degree  of its Gauss map $g^\sharp$ is $1$, then the total curvature is $-4\pi$. The classification of Osserman proves that $H_t(C)^\sharp$ is the Enneper surface  (\cite{oss}). Since this argument holds for any value of $t$,  we conclude that the quotient space of $\{H_t(C)^\sharp:t\in\r\}$ by congruences has only one element.

\begin{theorem}\label{t53} Consider $C$ the parabolic catenoid of axis $(1,0,1)$. 
The dual surface of $C$ by  the hyperbolic rotations $H_t$ of axis $(1,0,0)$   is the Enneper surface.
\end{theorem}

 \section*{Acknowledgements} The first author has been partially supported by  the MINECO/FEDER grant MTM2014-52368-P.   This paper was prepared while the second author was visiting the Departamento de Geometr\'{\i}a y Topolog\'{\i}a, Universidad de Granada. The second author wishes to thank this institution for its hospitality.
 %%%%%%%%%%%%


\begin{thebibliography}{99}


\bibitem{aa} Albujer, A. L., Al\'{\i}as, L. J.: Calabi-Bernstein results for maximal surfaces in Lorentzian product spaces. J. Geom. Phys. 59 (2009), 620--631.

\bibitem{acm} Al\'{\i}as, L. J., Chaves, R.M.B., Mira, P.: Bj\"{o}rling problem for maximal surfaces in Lorentz-Minkowski space. Math. Proc. Camb. Phil. Soc. 134, (2003), 289--316.

\bibitem{al}  Ara\'ujo, H., Leite, M. L.:  How many maximal surfaces do correspond to one minimal surface? Math. Proc. Cambridge Philos. Soc. 146 (2009), 165--175.

\bibitem{bo}   Bonnet, O.:  M\'emoire sur l'emploi d'un nouveau syst\`eme de variables dans l'\'etude des surfaces courbes. J. Math\'em.  Pures Appl.  2 (1860), 153--266.

 \bibitem{ca}  Calabi, E.:  Examples of Bernstein problems for some nonlinear equations. Proc. Symp. Pure Math. 15 (1970), 223–--230.
 
 \bibitem{co} Cho, J., Ogata, Y.: 
 Deformation of minimal surfaces with planar curvature lines. J. Geom. to appear.  
DOI: 10.1007/s00022-016-0352-0
  
 
 
\bibitem{go}  Goursat, E.: Sur un mode de transformation des surfaces minima. Acta Math., 11 (1887), 135--186.
 

\bibitem{gu} Gu, C.H.: The extremal surfaces in the 3-dimensional Minkowski space. Acta Math. Sin. New Ser. 1 (1985), 173--180.

\bibitem{udo}  Hertrich-Jeromin, U.: Introduction to M\"{o}bius differential geometry, volume 300 of London Mathematical Society Lecture Note Series. Cambridge University Press, Cambridge, 2003.
\bibitem{ko} Kobayashi, O.: Maximal surfaces in the 3-dimensional Minkowski Space $L^{3}$.  
 Tokyo J. Math. 6 (1983), 297--309.
 
 \bibitem{le1} Lee, H.:  Extensions of the duality between minimal surfaces and maximal surfaces. Geom. Dedicata 151 (2011), 373--386.  

\bibitem{le2} Lee, H.: 
Minimal surface systems, maximal surface systems and special Lagrangian equations.  
Trans. Amer. Math. Soc. 365 (2013),  3775--3797.

\bibitem{lei} Leite, M. L.: 
 Surfaces with planar lines of curvature and orthogonal systems of cycles. J. Math. Anal. Appl. 421 (2015), 1254--1273.  



 \bibitem{lls} L\'opez, F. J., L\'opez, R.,  Souam, R.:
Maximal surfaces of Riemann type in Lorentz-Minkowski space $L^3$ .
Michigan Math. J. 47 (2000),  469--497. 

\bibitem{ma} Manhart, F.: Bonnet-Thomsen surfaces in Minkowski geometry. J. Geom. 106 (2015), 47--61.
   


\bibitem{mp} Mira P., Pastor J. A.:  Helicoidal maximal surfaces in Lorentz-Minkowski space, Monatsh. Math.   140 (2003), 315--334.

 
 
\bibitem{ni}  Nitsche, J. C. C.: Lectures on Minimal Surfaces, vol. 1, Cambridge University Press, 1989.
 

%\bibitem{we} Webber, M.: Minimal Surfaces. Bloomington's Virtual Minimal Surface Museum. http://www.indiana.edu/~minimal/
 
\bibitem{oss}   Osserman, R.: A survey of Minimal Surfaces.  Cambridge
Univ. Press, New York, 1989.

\bibitem{pa} Palmer, B.: Spacelike constant mean curvature surfaces in pseudo-Riemannian space forms. Ann. Glob. Anal. Geom. 8 (1990), 217--226.

 



\end{thebibliography}
\end{document}